\documentclass[11pt,a4paper]{amsart}

\usepackage{amsmath,amsthm,amsfonts,amssymb,mathrsfs,graphicx,bbm,color}
\usepackage[shellescape,latex]{gmp}

\theoremstyle{plain}
\newtheorem{theorem}{Theorem}
\newtheorem{proposition}{Proposition}
\newtheorem{lemma}[theorem]{Lemma}

\newtheorem{remark}{Remark}

\DeclareMathOperator{\dimh}{dim_H}

\DeclareMathOperator{\E}{\mathsf{E}}
\newcommand{\Prob}{\mathsf{P}}

\DeclareMathOperator{\T}{\mathbb{T}}
\DeclareMathOperator{\UU}{\mathcal{U}}
\DeclareMathOperator{\EE}{\mathcal{E}_\mathrm{ah}}

\allowdisplaybreaks[3]

\title{Uniform random covering problems}

\author{Henna Koivusalo}
\address{School of Mathematics, University of Bristol, Fry Building, Woodland Road, Bristol BS8 1UG, UK}
\email{henna.koivusalo@bristol.ac.uk}

\author{Lingmin Liao}
\address{Univ Paris Est Creteil, CNRS, LAMA, F-94010 Creteil,
  France \& Univ Gustave Eiffel, LAMA, F-77447 Marne-la-Vall\'ee,
  France} \email{lingmin.liao@u-pec.fr}

\author{Tomas Persson}
\address{Centre for Mathematical Sciences, Lund University, Box
  118, 221 00 Lund, Sweden}
\email{tomasp@maths.lth.se}

\thanks{We thank Mark Holland for informing us about the reference \cite{Galambos}}

\begin{document}

\begin{abstract}
  Motivated by the random covering problem and the study of
  Dirichlet uniform approximable numbers, we investigate the uniform
  random covering problem. Precisely, consider an i.i.d.\
  sequence $\omega=(\omega_n)_{n\geq 1}$ uniformly distributed on the unit
  circle $\mathbb{T}$ and a sequence
  $(r_n)_{n\geq 1}$ of positive real numbers with limit $0$. We investigate the
  size of the random set
  \[
  \UU (\omega):=\{y\in \mathbb{T}: \ \forall N\gg 1,
  \ \exists n \leq N, \ \text{s.t.} \ \| \omega_n -y \| < r_N \}.
  \]
  Some sufficient conditions for $\UU(\omega)$ to be almost
  surely the whole space, of full Lebesgue measure, or countable,
  are given. In the case that $\UU(\omega)$ is a Lebesgue
  null measure set, we provide some estimations for the upper and
  lower bounds of Hausdorff dimension.
\end{abstract}

\date{\today}

\maketitle

\section{Introduction}

Let $\mathbbm{T}=\mathbbm{R}/\mathbbm{Z}$ be the one dimensional
torus. Denote by $\| \cdot \|$ the distance of a point in
$\mathbbm{T}$ to the point $0$.  The famous Dirichlet Theorem
states that for any real numbers $\theta$ and $N\geq 1$, there
exists an integer $1\leq n \leq N$, such that $\|n\theta\| <
N^{-1}$. As corollary, for any real number $\theta$, there exists
infinitely many integers $n$, such that $\| n \theta \|
<n^{-1}$. The Dirichlet Theorem and its corollary tell us that
for any $\theta$, in two different ways, $0$ is approximated by
the sequence $n\theta$ with degree-one-polynomial speed. Such two
different ways are called uniform approximation (uniform with
respect to $N$) and asymptotic approximation in the survey paper \cite{Wa12} of Waldschmidt.

In general, one can study the approximation of any point $y$ by
the sequence $n \theta$ with a faster speed.
%Historically, the asymptotic approximation had received much more attention than the uniform approximation.  and had became a main topic of research in Diophantine approximation and metric number theory. One is referred to the books of Schmidt and Harman for more details. As an example of the research in asymptotic approximation, one could cite the result of 
For the asymptotic approximation, in 2003, Bugeaud \cite{Bu03},
and independently, Schmeling and Troubetzkoy \cite{ST03} proved
that for any irrational $\theta$, for any $\alpha>1$, the
Hausdorff dimension of the set
\[
\{\, y\in \mathbbm{T}: \| n\theta -y \| < n^{-\alpha} \ \text{for
  infinitely many } n \,\}
\]
is $1/\alpha$. The corresponding uniform approximation problem
was quite recently studied by Kim and Liao \cite{KL19} who proved
that the Hausdorff dimension of the set
\[
\UU[\theta, \alpha] := \{\, y\in \mathbbm{T}: \ \forall
N\gg 1, \ \exists n \in [1, N], \ \text{s.t.} \ \| n\theta -y \|
< N^{-\alpha} \, \}
\]
%\[
%\{y\in \mathbbm{T}: \ \text{for all large } N \ \text{there exists } 1\leq n \leq N, \ \text{such that} \ \| n\theta -y \| < N^{-\tau}  \}
%\]
depends on the irrationality exponent of $\theta$ defined by $ w(\theta) := \sup \{ s > 0 : \liminf_{j \to \infty} j^{s} \| j \theta \| = 0 \}.$  Specially,
when $w(\theta)=1$ (thus for
Lebesgue almost all $\theta$), the Hausdorff dimension of
$\UU[\theta, 1]$ is between $1/2$ and $1$. For the
complicated dimensional formulae and estimations, one can consult
\cite{KL19}.

Motivated by these metric number theory results, one wonders
about the analog results when the sequence $n \theta$ is replaced
by an i.i.d.\ sequence.
% which is uniformly distributed on the torus $\mathbbm{T}$.
In fact, for the asymptotic approximation, this is nothing but the widely
studied Dvoretzky covering problem. Let $(\omega_n)_{n\geq 1}$ be an
i.i.d.\ random sequence of uniform distribution on the unit
circle $\mathbbm{T}$. Let $(r_n)_{n\geq 1}$ be a decreasing
sequence of positive real numbers with $\sum_{n=1}^\infty
r_n=\infty$. In 1956, Dvoretzky \cite{Dr56} asked what are
necessary and sufficient conditions on $(r_n)_{n\geq 1}$ such
that almost surely, all points in $\mathbbm{T}$ are covered by
infinitely many open intervals with center $\omega_n$ and radius
$r_n$, %to cover the whole circle $\mathbbm{T}$
or equivalently, %to have 
\begin{align}\label{prob-covering}
  \Prob \big(\{\, y\in \mathbbm{T}: \| \omega_n -y \| < r_n \ \text{for
    infinitely many } n \,\}=\mathbbm{T}\big)=1.
\end{align}
This problem attracted much attention of mathematicians, such as
L\'evy, Kahane, Erd\H{o}s, Billard, et al.\ (see Kahane's book
\cite{Ka85} and his survey paper \cite{Ka00}).  Specially, for
the case $r_n=c/n$ ($c>0$), Kahane \cite{Ka59} proved in 1959
that \eqref{prob-covering} holds when $c>1$.  In 1961, Erd\H{o}s
\cite{Er61} announced that \eqref{prob-covering} holds if and
only if $c\geq 1$, but never published a proof. In 1965, Billard
\cite{Bi65} showed that \eqref{prob-covering} does not hold if
$c<1$. Finally, Orey \cite{Or71} in 1971 and independently
Mandelbrot \cite{Ma72} in 1972, proved that \eqref{prob-covering}
holds if $c = 1$. The complete solution to the Dvoretzky problem
was given in 1972 by Shepp \cite{Shepp} who proved that
\eqref{prob-covering} holds if and only if
\[
\sum_{n=1}^\infty \frac{1}{n^2} \exp(r_1 + \ldots + r_n)=\infty.
\] 

When $r_n$ decreases to $0$ faster, one is also interested in the
Hausdorff dimension of the set of points which are covered
infinitely often by the random intervals. In 2004, Fan and Wu
\cite{FW04} proved that almost surely, the Hausdorff dimension of
the set
\[
\{y\in \mathbbm{T}: \| \omega_n -y \| < n^{-\alpha} \ \text{for
  infinitely many } n \}
\]
is $1/\alpha$ for all $\alpha \geq 1$. Comparing with the above
mentioned result of Bugeaud and Schmeling--Troubetzkoy, one finds
that the i.i.d.\ sequence exhibits some similar asymptotic approximation
property as the irrational rotation sequence $n \theta$.

As a counter part of the famous random covering problem which
corresponds to the asymptotic Diophantine approximation, we would
like to study the uniform covering problem which corresponds to
the uniform Diophantine approximation. Analogously, for an
i.i.d.\ random sequence $\omega=(\omega_n)_{n\geq 1}$ of uniform
distribution and a real positive sequence
$(r_n)_{n\geq 1}$, we want to describe the size (in the sense
of Lebesgue measure and Hausdorff dimension) of the random set
\[
\UU(\omega) := \{\, y\in \mathbbm{T}: \ \forall N\gg 1,
\ \exists n \leq N, \ \text{s.t.} \ \| \omega_n -y \| < r_N \,\}.
\]
If we let $B_{k,n} = B (\omega_k, r_n)$
and
\[
E_n = \bigcup_{k=1}^n B_{k,n},
\]
then
\[
\UU(\omega) = \liminf_{n \to \infty} E_n.
\]

\medskip
\section{Results}

Our first main theorem gives a sufficient condition and a
necessary condition for $\Prob (\UU(\omega)=\mathbbm{T}) = 1$.

\begin{theorem}\label{the:cover}

(i)  If $\sum\limits_{n=1}^\infty n (1-r_n)^{n}<\infty$, then almost
  surely $\UU(\omega) = \mathbbm{T}$.
  
  In particular, if $r_n = \frac{c \log n}{n}$ and $c > 2$, then
  almost surely $\UU(\omega) = \mathbbm{T}$. 

\medskip
(ii)  If $\liminf\limits_{n \to \infty} n (1-r_n)^n > 0$, then with positive
  probability, $\UU (\omega) \neq \mathbb{T}$.

  Furthermore, if $r_n = \frac{c \log n}{n}$ and $c < 1$, then
  $\UU (\omega) \neq \mathbb{T}$ almost surely.
\end{theorem}

\begin{remark}
  Note that the condition $\sum\limits_{n=1}^\infty n (1-r_n)^n < \infty$ also holds
  for $r_n = \frac{2 \log n+\gamma\log\log n}{n}$,
  $\gamma > 1$.
\end{remark}

As for the Lebesgue measure, we have the following theorem.

\begin{theorem}\label{the:measure} 
 (i) Assume that $(r_n)_{n\geq 1}$ is decreasing. Then either $\Prob (\lambda
  (\UU (\omega)) = 1) = 1$ or $\Prob (\lambda
  (\UU (\omega)) = 0) = 1$, depending on $(r_n)_{n\geq 1}$.

\smallskip
  (ii) Suppose that the sequences $(r_n)_{n\geq 1}$ and $(nr_n)_{n\geq 1}$ are
  decreasing.  We have $\lambda (\UU (\omega)) = 1$ almost
  surely, if and only if the sequence $(r_n)_{n\geq 1}$ satisfies
  \begin{equation} \label{eq:galambos}
   \sum_{n=1}^\infty r_n = \infty \qquad \text{and} \qquad \sum_{n=1}^\infty r_n e^{- 2 n r_n} < \infty.
  \end{equation}
  In particular, if $r_n = \frac{c \log \log n}{2n}$ then
  $\lambda (\UU (\omega)) = 1$ almost surely if and only if $c >
  1$.
\end{theorem}

\begin{remark}
  Condition \eqref{eq:galambos} holds also for $r_n = \frac{\log
    \log n+ \gamma \log\log\log n}{2n}$ with $\gamma>1$.
\end{remark}

We also give a sufficient condition that $\UU(\omega)$ is
countable.

\begin{theorem} \label{the:countable}
  If $\sum\limits_{n=1}^\infty n r_n < \infty$, then almost surely
  $\UU(\omega) = \{\, \omega_k : k \in \mathbbm{N} \,\}$.
\end{theorem}

Finally, some estimations of the Hausdorff dimension of
$\UU(\omega)$ are obtained in the following two
theorems.

%\begin{theorem}\label{the:dim0}
%  If $r_n = 1/n^\alpha$ with $\alpha > 1$, then $\dimh \UU(\omega) = 0$
%  almost surely.
%\end{theorem}

\begin{theorem} \label{the:dimupperbound}
  If $r_n = \frac{c}{n}$ and $0 < c < \frac{1}{2}$ then  almost surely
  \[
  \dimh \UU(\omega) \leq \inf_{\theta > 1} \frac{\log \Lambda}{\log
    \theta},
  \]
 where
  \[
  \Lambda = \frac{1 + \Theta + \Delta}{2} + \sqrt{ \Bigl( \frac{1
      + \Theta + \Delta}{2} \Bigr)^2 - \Delta},
  \]
  and
  \begin{align*}
    \Theta &= 2c (\theta - 1) (1 + \theta^{-2}), \\ \Delta &= 2c
    (\theta - 1) (\theta^{-1} - \theta^{-2}).
  \end{align*}
  
  In particular, since $\inf_{\theta > 1} \frac{\log
    \Lambda}{\log \theta}$ tends to $0$ as $c \to 0$, we conclude
  that $\dimh \UU(\omega) = 0$ almost surely when $r_n =
  1/n^\alpha$ with $\alpha > 1$.
\end{theorem}

\begin{theorem} \label{the:dimlowerbound}
  If $r_n = \frac{c}{n}$, then almost surely
  \[
  \dimh \UU(\omega) \geq 1 - \frac{- \log \bigl(1 - \exp \bigl(-2c
    \frac{\theta-1}{\theta^2} \bigr) \bigr)}{\log \theta}.
  \]

  In particular (let $\theta = 8.6$), if $r_n = \frac{1}{n}$,
  then almost surely
  \[
  \dimh \UU(\omega) \geq 0.2177444298485995.
  \]
\end{theorem}

Below is an illustration of the dimension bounds provided by
Theorems~\ref{the:dimupperbound} and
\ref{the:dimlowerbound}. Note that there is no $c$ for which the
dimension is known to be intermediate, but at least
Theorems~\ref{the:dimupperbound} and \ref{the:dimlowerbound} show
that $r_n = \frac{c}{n}$ is the ``right'' quantity to look at, in the sense that for such sequences $(r_n)_{n\geq 1}$ there is a chance for the dimension to be intermediate.

\begin{center}
  \includegraphics{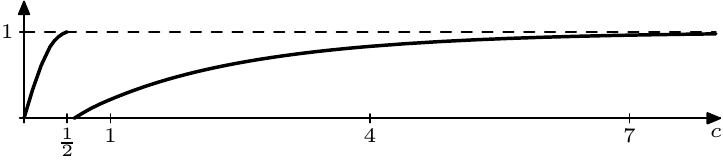}
\end{center}

The proofs of the above results are given in
Sections~\ref{sec:cover}--\ref{sec:dimension}.

\medskip
\section{Open questions and problems}

Our results do not give a complete picture of the size of the set
$E$. We list below some open questions and problems.

\begin{enumerate}
\item Is there a zero--one law for the event $\UU (\omega) =
  \mathbbm{T}$? What is a necessary and sufficient condition on
  $(r_n)_{n\geq 1}$ for $\Prob (\UU (\omega) = \mathbbm{T}) = 1$?
  
\item Is there a zero--one law for the Hausdorff dimension of
  $\UU(\omega)$? Is the probability $\Prob (\dimh
  \UU(\omega) = s)$ ($0\leq s \leq 1$) always equal to $0$ or $1$?

\item Give better dimension estimates of $\dimh \UU(\omega)$ when
  $r_n = \frac{c}{n}$. In particular, is there a value of $c$
  such that $\Prob (0 < \dimh \UU(\omega) < 1) > 0$?
\end{enumerate}

\medskip
\section{Proof of Theorem~\ref{the:cover} on uniform covering} \label{sec:cover}

In this section, we prove Theorem~\ref{the:cover} which gives
sufficient conditions for $\UU(\omega) =\mathbbm{T}$ to hold almost
surely.

By Shepp \cite[Formula~(90)]{Shepp}, we have
\begin{align*}
  \Prob (\mathbbm{T} \not \subset E_n) &\leq \frac{2 (1 -
    r_n)^{2n}}{\int_0^{r_n} (1 - r_n - t)^n \, \mathrm{d} t +
    (\frac{1}{4} - r_n) (1 - 2 r_n)^n} \\ & = \frac{2 (1 -
    r_n)^{2n}}{\frac{(1-r_n)^{n+1} - (1-2r_n)^{n+1}}{n+1} +
    (\frac{1}{4} - r_n) (1 - 2 r_n)^n}.
\end{align*}
For large enough $n$, we therefore have
\[
\Prob (\mathbbm{T} \not \subset E_n) \leq \frac{2 (1 -
  r_n)^{2n}}{\frac{1}{n+1}(1-r_n)^{n+1}} = 2 (n+1) (1-r_n)^{n-1}.
\]
Thus, if $\sum_{n=1}^\infty n (1-r_n)^{n}<\infty$, then $\Prob
(\mathbbm{T} \not \subset E_n)$ is summable. By Borel--Cantelli Lemma,
it follows that almost surely, $E_n = \mathbbm{T}$ for all but
finitely many $n$.
 
For the special case $r_n
= c \log n / n$ ($c > 2$), one can easily check that
$\sum_{n=1}^\infty n (1-r_n)^{n}<\infty$. The first part of the
theorem is thus proved.
 
For the second part, we deduce from Shepp\footnote{There is a misprint
  in (91) of \cite{Shepp}. ``$U(\alpha) \not \subset C$'' should be
  ``$C \not \subset U(\alpha)$''.} \cite[Formula~(91)]{Shepp} that
\begin{align*}
  \Prob (\mathbbm{T} \not \subset E_n) &\geq \frac{1}{2} \frac{(1
    - r_n)^{2n}}{\int_0^{r_n} (1 - r_n - t)^n \, \mathrm{d} t +
    (\frac{1}{2} - r_n) (1 - 2 r_n)^n} \\ & = \frac{1}{2}
  \frac{(1 - r_n)^{2n}}{\frac{(1-r_n)^{n+1} -
      (1-2r_n)^{n+1}}{n+1} + (\frac{1}{2} - r_n) (1 - 2 r_n)^n}.
\end{align*}
Hence
\[
\frac{1}{\Prob (\mathbbm{T} \not \subset E_n)} \leq 2 \frac{1}{(n+1)
  (1 - r_n)^{n-1}} + \frac{(1 - 2r_n)^n}{(1 - r_n)^{2n}}.
\]
Remark that we always have $\frac{(1 - 2r_n)^n}{(1 - r_n)^{2n}} \leq 1$. Thus
\[
\frac{1}{\Prob (\mathbbm{T} \not \subset E_n)} \leq 2 \frac{1}{(n+1)
  (1 - r_n)^{n-1}} + 1.
\]

Assuming that $r_n \to 0$, we obtain 
\[
\limsup_{n \to \infty} \Prob (\mathbbm{T} \not \subset E_n) \geq
\frac{1}{1 + \frac{2}{\liminf_{n \to \infty} n (1 - r_n)^n}}.
\]
Hence, if $\liminf\limits_{n \to \infty} n (1 - r_n)^n = p > 0$, then
\[
\Prob (\mathbbm{T} \not \subset \UU (\omega)) = \Prob (\limsup_{n
  \to \infty} \{ \mathbbm{T} \not \subset E_n \}) \geq \limsup_{n
  \to \infty} \Prob (\mathbbm{T} \not \subset E_n) \geq
  \frac{1}{1 + 2/p} > 0.
\]

If $r_n = \frac{c \log n}{n}$ with $c < 1$, then $\liminf\limits_{n \to
  \infty} n (1 - r_n)^n = \infty$. Hence $$\Prob (\mathbbm{T} \not \subset
\UU (\omega)) = 1.$$ Note also that with $c = 1$ we have $\liminf\limits_{n \to
  \infty} n (1 - r_n)^n = 1$. Hence $$\Prob (\mathbbm{T} \not \subset \UU
(\omega)) \geq \frac{1}{3}.$$

\section{Proof of Theorem~\ref{the:measure} on Lebesgue measure}

\subsection{Proof of the zero--one law}

The i.i.d.\ sequence $\omega=(\omega_1, \omega_2, \dots)$ can be
naturally modeled as an element in the probability space
$\Omega:=\mathbbm{T}^\mathbbm{N}$ with the $\sigma$-algebra being the
infinite product $\sigma$-algebra of the Borel $\sigma$-algebra
on $\mathbbm{T}$, and the probability $\Prob$ being the infinite
product of the Lebesgue measure on $\mathbbm{T}$. Then by defining
$T$ as the left shift on $\Omega$, we know that $T$ is an ergodic
transformation with respect to $\Prob$.

Note that for any $\omega\in \Omega$ and any point $y\in
\mathbbm{T}$, we have $y\in \UU(\omega)$ if any only if
\[
\forall N\gg 1, \ \exists n \leq N, \ \text{s.t.} \ \omega_n \in
B(y, r_N),
\]
or, equivalently, $\omega$ is in the following set
\[
\bigcup_{p=1}^\infty\bigcap_{n=p}^\infty \big\{\, \omega :
\{\omega_1, \dots, \omega_n\} \cap B(y, r_n)\neq \emptyset \,\big\}.
\]

For $y\in \mathbbm{T}$ and $(r_n)_{n\geq 1}$, let 
\[
B_n(y):= B(y, r_n)\times \T^{\mathbbm{N}}.
\]
Then $(B_n(y))_{n\geq 1}$ is a sequence of shrinking targets such
that $\Prob (B_n(y)) \to 0$ as $n \to \infty$. Further, $y\in
\UU(\omega)$ if any only if
\[
\forall n\gg 1, \ \exists k \leq n, \ \text{s.t.} \ T^{k-1}\omega
\in B_n (y),
\]
which is equivalent to 
\[
\omega \in \bigcup_{p=1}^\infty\bigcap_{n=p}^\infty
\bigcup_{k=0}^{n-1} T^{-k} B_n (y)=:\mathcal{E}_{\rm ah}(y).
\]
By \cite[Lemma~1]{KKP20}, for fixed $y$, the set $\EE(y)$ has
probability $1$ or $0$. Because of rotational invariance, $\Prob
(\EE (y))$ does not depend on $y$ but only on the sequence
$(r_n)_{n\geq 1}$. Hence, by Fubini's theorem, we have either  $\Prob
(\lambda (\UU (\omega)) = 1) = 1$ or $\Prob (\lambda
(\UU (\omega)) = 0) = 1$, which proves the first part of
Theorem~\ref{the:measure}.

\subsection{The condition on $\boldsymbol{(r_n)_{n\geq 1}}$}

That the condition \eqref{eq:galambos} in
Theorem~\ref{the:measure} is necessary and sufficient for
$\lambda (\UU (\omega)) = 1$ to hold almost surely follows from Theorem~4.3.1 of the book  \cite{Galambos} of Galambos. We present here a simplified statement. 

\begin{theorem}[Theorem 4.3.1 of Galambos \cite{Galambos}]
Let $X_1, X_2, \dots$ be a sequence of independent, identically distributed random variables with a nondegenerate, continuous distribution function $F$. Let $Z_n=\max \{X_1, X_2, \ldots, X_n\}$, and assume that sequences $(u_n)_{n\geq 1}$ and $(n(1-F(u_n)))_{n\geq 1}$ are both increasing. Then the probability
\begin{align}\label{formula-G}
\Prob(Z_n\le u_n \ \text{\rm for infinitely many } n)=0,
\end{align}
if and only if  
\[
\sum_{j=1}^\infty (1-F(u_j))=\infty \quad \text{and}\quad \sum _{j=1}^\infty (1-F(u_j))\exp(-j(1-F(u_j))<\infty. 
\]
\end{theorem}

We will now connect the quantities in our special case to the notation of Galambos. 

Fix a point $y \in \mathbbm{T}$. Let $X_n = |\omega_n - y|^{-1}$
and $u_n = r_n^{-1}$. The sequence of random variables $(X_n)_{n\geq 1}$ is
i.i.d.\ and $y \in B(\omega_k, r_n)$ if and only if $X_k > u_n$.

Notice that $Z_n > u_n$
if and only if there is a $k \leq n$ such that $y \in B(\omega_k,
r_n)$. Thus, $\omega_k \in \EE (y)$ if and only if $Z_n > u_n$
holds eventually (i.e.\ for all sufficiently large $n$). Hence,
$\omega_k \in \EE (y)$ if and only if $Z_n \leq u_n$ holds for at
most finitely many $n$.

We have $F(x) = \Prob (X_n < x) = \Prob (|\omega_n - y| > x^{-1})
= 1 - 2 x^{-1}$. By the above theorem of Galambos, we
have (\ref{formula-G}) holds  %$\Prob (Z_n \leq u_n \text{ i.o.}) = 0$
 if and only if
\[
\sum_{n=1}^\infty \Prob (X_n > u_n) = \infty \quad \text{and}\quad\sum_{n=1}^\infty (1 - F(u_n)) \exp ( -n (1 - F(u_n))) <
\infty.
\]
This translates immediately to condition \eqref{eq:galambos} in
Theorem~\ref{the:measure}. Therefore, $\Prob (\EE (y)) =1
$ if and only if \eqref{eq:galambos} holds. Since this holds for
all $y$, it follows by Fubini's theorem that $\Prob (\lambda(\UU (\omega)) = 1) = 1$
if and only if \eqref{eq:galambos} holds.

\smallskip
Finally, we note that the zero--one law of
Theorem~\ref{the:measure} alternatively can also be deduced from \cite[Lemma 4.3.1]{Galambos} instead of referring to
\cite[Lemma~1]{KKP20}.

\medskip
\section{Proof of Theorem~\ref{the:countable}}

The probability that $B_{n+1,n+1}$ intersects no ball of
$B_{k,n}$ for $k \leq n$ is at least
\[
1-2n(r_n+r_{n+1}).
\]
Thus, the probability that $B_{n+1,n+1}$ intersects at least one
of the balls $B_{k,n}$ for some $k \leq n$ is at most $2n(r_n+
r_{n+1})$. Hence, if $\sum_{n=1}^\infty n r_n < \infty$, then almost surely,
there is an $m$ such that for all $n \geq m$ the ball
$B_{n+1,n+1}$ has empty intersection with all balls $B_{k,n}$
with $k \leq n$. Therefore, for all $n\geq m$,
\begin{align*}
  E_n \cap E_{n+1}&=\bigcup_{k=1}^n B_{k,n} \cap \left(
  \bigcup_{k=1}^n B_{k,n+1} \cup B_{n+1, n+1}\right)
  \\ &=\left(\bigcup_{k=1}^n B_{k,n} \cap \bigcup_{k=1}^n
  B_{k,n+1}\right) \cup \left(\bigcup_{k=1}^n B_{k,n} \cap
  B_{n+1, n+1}\right)\\ &=\bigcup_{k=1}^n B_{k,n+1}.
\end{align*}
Further, we have 
\[
\bigcap_{n=p}^\infty E_n=\{\omega_1, \dots, \omega_p\}, \quad \forall
p\geq m,
\]
which implies 
\[
\UU(\omega) =\bigcup_{p=1}^\infty\bigcap_{n=p}^\infty E_n= \{\,
\omega_k : k \in \mathbbm{N} \,\}.
\]

\medskip
\section{Proofs related to Hausdorff dimension} \label{sec:dimension}

In this section we prove the theorems on the estimations of the
Hausdorff dimension of the set $\UU(\omega)$.

\subsection{Proofs of upper bounds on Hausdorff dimension}

Before we give the proof of the upper bound of the Hausdorff
dimension which is found in Theorem~\ref{the:dimupperbound}, we
give a theorem with a somewhat weaker upper bound. We include the proof of
this theorem since it follows the same lines of thought as the more difficult proof of
Theorem~\ref{the:dimupperbound}, and might make the proof of Theorem \ref{the:dimupperbound} easier to read. 

\begin{theorem} \label{the:dimupperboundold}
  If $r_n = \frac{c}{n}$ and $0 < c < \frac{1}{2}$, then
  \[
  \dimh \UU(\omega)  \leq \inf_{\theta > 1} \frac{\log \bigl(1 + 2 c 
    \frac{\theta^2 - 1}{\theta} \bigr)}{\log \theta}
  \]
  almost surely.
\end{theorem}

\begin{proof}
  Put $n_j = \theta^j$ and let $l > 0$.
  Consider the set
  \[
  G_{l,i} = \bigcap_{j=l}^{i} \bigcup_{k=1}^{n_j} B(\omega_k, r_{n_j}).
  \]
  We are going to construct inductively a cover of $G_{l,i}$ by
  $N_i$ balls $B(\omega_k, r_{n_i})$, where $k \in I_i$ and $I_i
  \subset \{1,2, \ldots, n_i \}$.

  For $i = l$, we let $I_l = \{1, 2, \ldots, n_l\}$. Suppose that
  $I_i$ has been defined. We define $I_{i+1}$ to consist of those
  $k \leq n_{i+1}$ such that $B(\omega_k, r_{n_{i+1}})$ intersects the
  set
  \[
  \hat{G}_i = \bigcup_{j \in I_i} B(\omega_j, r_{n_i}).
  \]
  
  Since $B(\omega_k, r_{n_{i+1}})$ is contained in $B(\omega_k,
  r_{n_{j}})$ for all $j \leq i$, we have
  \[
  N_{i+1} = N_i + M_{i+1},
  \]
  where $M_{i+1}$ is the number of $n_i < k \leq n_{i+1}$ such
  that the ball $B(\omega_k, r_{n_{i+1}})$ has non-empty intersection
  with $\hat{G}_i$.

  Let $(\hat{G}_i)_{(r)}$ denote the $r$-neighbourhood of
  $\hat{G}_i$. We then have
  \begin{equation} \label{eq:Mformula}
    M_{i+1} = \sum_{k = n_i + 1}^{n_{i+1}}
    \mathbbm{1}_{(\hat{G}_i)_{(r_{n_{i+1}})}} (\omega_k).
  \end{equation}

  Let $\varepsilon > 0$. The set $\hat{G}_i$ is a union of $N_i$
  balls of radius $r_{n_i}$. Hence, the Lebesgue measure of
  $(\hat{G}_i)_{(r_{n_{i+1}})}$ is at most $2 (r_{n_i} +
  r_{n_{i+1}}) N_i$.  It follows that
  \[
  \E ( M_{i+1} | \mathscr{S}_i ) \leq 2 (r_{n_i} + r_{n_{i+1}})
  N_i (n_{i+1} - n_i) = 2c \frac{\theta^2 - 1}{\theta} N_i,
  \]
  where $\mathscr{S}_i$ denotes the $\sigma$-algebra generated by
  $\omega_1, \omega_2, \ldots, \omega_{n_i}$. Hence
  \[
  \E M_{i+1} = \E ( \E ( M_{i+1} | \mathscr{S}_i )) \leq 2c
  \frac{\theta^2 - 1}{\theta} \E N_i.
  \]
  Since $N_{i+1} = N_i + M_{i+1}$, it follows that
  \[
  \E N_{i+1} \leq \Bigl( 1 + 2c \frac{\theta^2 - 1}{\theta}
  \Bigr) \E N_i.
  \]
  By induction,
  \[
  \E N_{i+1} \leq \Bigl( 1 + 2c \frac{\theta^2 - 1}{\theta}
  \Bigr)^{i+1-l} \E N_l = \Bigl( 1 + 2c \frac{\theta^2 -
    1}{\theta} \Bigr)^{i+1-l} n_l.
  \]

  By Markov's inequality,
  \[
  \Prob \{ N_{i+1} \geq u_i \E N_{i+1} \} \leq \frac{1}{u_i}.
  \]
  Letting $u_i = (1 + \varepsilon)^i$ for some $\varepsilon > 0$,
  we therefore have 
  \begin{align*}
    \Prob \Bigl\{ N_{i+1} \geq (1 + \varepsilon)^i \Bigl( 1 + 2c
    \frac{\theta^2 - 1}{\theta} \Bigr)^{i+1-l} n_l \Bigr\} &\leq
    \Prob \big\{ N_{i+1} \geq u_i \E N_{i+1} \big\} \\ &\leq (1 +
    \varepsilon)^{-i},
  \end{align*}
  which is summable over $i$. Hence, almost surely, there is an
  $i_0$ such that
  \begin{align*}
    N_{i+1} &\leq (1 + \varepsilon)^i \Bigl( 1 + 2c
    \frac{\theta^2 - 1}{\theta} \Bigr)^{i+1-l} n_l \\ &= C_l (1 +
    \varepsilon)^{i+1} \Bigl( 1 + 2c \frac{\theta^2 - 1}{\theta}
    \Bigr)^{i+1}
  \end{align*}
  holds for all $i \geq i_0$. We assume from now on that such an
  $i_0$ exists.
  
  With $i > i_0$, we may cover the set
  \[
  G_{n_l} = \bigcap_{j=n_l}^\infty \bigcup_{k = 1}^{n_j} B(\omega_k,
  r_{n_j})
  \]
  by $N_i$ balls of radius $r_{n_i}$. Hence
  \[
  \dimh G_{n_l} \leq \frac{\log \Bigl( (1 + \varepsilon) \bigl(1
    + 2 c \frac{\theta^2 - 1}{\theta} \bigr) \Bigr)}{\log
    \theta}.
  \]
  Since $E$ is contained in the union of the sets $G_{n_l}$, we
  have
  \[
  \dimh E \leq \frac{\log \Bigl( (1 + \varepsilon) \bigl(1 + 2 c
    \frac{\theta^2 - 1}{\theta} \bigr) \Bigr)}{\log \theta},
  \]
  and since $\varepsilon$ can be taken as small as we please, the
  theorem is proved.
\end{proof}

We now give the proof of Theorem~\ref{the:dimupperbound},
which contains a more careful but similar analysis as in the proof of
Theorem~\ref{the:dimupperboundold}.

\begin{proof}[Proof of Theorem~\ref{the:dimupperbound}]
  The proof is similar to that of
  Theorem~\ref{the:dimupperboundold}.  We let $n_j = \theta^j$
  and $l > 0$. As before, we construct inductively a cover of
  \[
  G_{l,i} = \bigcap_{j=l}^i \bigcup_{k = 1}^{n_j} B(\omega_k,
  r_{n_j}).
  \]
  The cover will consist of $N_i + Q_i$ balls $B (\omega_k,
  r_{n_i})$. We let $I_i, J_i \subset \{1,2,\ldots, n_i\}$ be
  such that $I_i \cap J_i = \emptyset$, there are $N_i$ balls
  $B(\omega_k,r_{n_i})$ with $k \in I_i$ and there are $Q_i$
  balls $B(\omega_k, r_{n_i})$ with $k \in J_i$. The construction
  of $I_i$ and $J_i$ is described below.

  We let $I_l = \{1, 2, \ldots, n_l\}$ and $J_l =
  \emptyset$. Hence $N_l = n_l$ and $Q_l = 0$. The set $G_{l,l}$
  is covered by the $N_l = n_l$ balls $B(\omega_k, r_{n_l})$,
  where $k \in I_l$.

  Suppose that the cover of $G_{l,i}$ is defined for some $i$,
  that is that
  \[
  G_{l,i} \subset \bigcup_{k \in I_i \cup J_i} B(\omega_k, r_{n_i}).
  \]
  The balls counted by $Q_i$ are balls $B(\omega_k, r_{n_i})$ such
  that $B(\omega_k, r_{n_{i+1}})$ can be discarded in the cover of
  $G_{l,i+1}$. The $N_i$ balls are balls such that $B(\omega_k,
  r_{n_{i+1}})$ is not discarded in the cover of $G_{l,i+1}$
  (regardless of whether they \emph{can} be discarded or
  not). What determines if $k \in I_{i+1}$ or $k \in J_{i+1}$ is
  described below.

  Consider first a $k \in \{1,2,\ldots, n_i\}$. If $k \in I_i$,
  then we let $k \in I_{i+1}$. Otherwise $k$ is not included in
  $I_i$ or $J_i$. This means that all $k \in J_i$ are discarded
  for the next step.

  We now consider a $k \in \{n_i+1, n_i+2,\ldots,n_{i+1}\}$. If
  $B(\omega_k,r_{n_{i+1}})$ intersects the set
  \[
  H_i = \bigcup_{k \in I_i \cup J_i} B(\omega_k, r_{n_i}),
  \]
  then we include $k$ in either $I_{i+1}$ or $J_{i+1}$. If
  $B(\omega_k,r_{n_{i+1}}) \cap H_i = \emptyset$, then $k$ is not
  included in any of $I_{i+1}$ or $J_{i+1}$.

  Suppose that $B(\omega_k,r_{n_{i+1}}) \cap H_i \neq
  \emptyset$. Then there exists an $l \in I_i \cup J_i$ such that
  $B(\omega_k,r_{n_{i+1}}) \cap B(\omega_l, r_{n_i}) \neq
  \emptyset$. Hence $|\omega_k - \omega_l| < r_{n_i} + r_{n_{i+1}}$.

  If $|\omega_k - \omega_l| > r_{n_{i}} + r_{n_{i+2}}$, then
  $B(\omega_k, r_{n_{i+2}}) \cap B(\omega_l, r_{n_{i}}) =
  \emptyset$. Suppose that $|\omega_k - \omega_l| > r_{n_{i}} +
  r_{n_{i+2}}$ holds for all $l\in I_i\cup J_i$. Then
  $B(\omega_k, r_{n_{i+2}})$ will have empty intersection with
  $H_i$ and it is therefore not necessary to include $k$ in any
  of $I_{i+2}$ or $J_{i+2}$. We therefore put $k$ in $J_{i+1}$ in
  this case.

  Finally, if $k$ satisfies $|\omega_k - \omega_l| < r_{n_{i}} +
  r_{n_{i+2}}$ for some $l$, then we include $k$ in $I_{i+1}$. In
  this way we obtain 
  \[
  H_{i+1} = \bigcup_{k \in I_i \cup J_i} B(\omega_k, r_{n_i}) \supset
  G_{l,i+1}
  \]
  and by induction $H_i \supset G_{l,i}$ for all $i$.
  
  As before, we let $\mathscr{S}_i$ denote the $\sigma$-algebra
  generated by $\omega_1, \omega_2, \ldots, \omega_{n_i}$. We get
  \[
  \left\{
  \begin{array}{lcl}
    \E (N_{i+1} | \mathscr{S}_i) & \leq & N_i + 2 (r_{n_i} +
    r_{n_{i+2}}) (n_{i+1} - n_i) (N_i + Q_i), \\ \E (Q_{i+1}|
    \mathscr{S}_i) & \leq & 2 (r_{n_{i+1}} - r_{n_{i+2}})
    (n_{i+1} - n_i) (N_i + Q_i).
  \end{array} \right.
  \]
  Hence
  \[
  \left\{
  \begin{array}{lcl}
    \E N_{i+1} & \leq & \E N_i + 2 (r_{n_i} + r_{n_{i+2}})
    (n_{i+1} - n_i) (\E N_i + \E Q_i), \\ \E Q_{i+1} & \leq & 2
    (r_{n_{i+1}} - r_{n_{i+2}}) (n_{i+1} - n_i) (\E N_i + \E
    Q_i).
  \end{array} \right.
  \]
   
  Letting
  \begin{align*}
    \Theta &= 2c (\theta - 1) (1 + \theta^{-2}),
    \\ \Delta &= 2c (\theta - 1) (\theta^{-1} - \theta^{-2}),
  \end{align*}
  we have
  \[
  \begin{bmatrix}
    \E N_{i+1} \\ \E Q_{i+1}
  \end{bmatrix} \le
  \begin{bmatrix}
    1 + \Theta & \Theta \\ \Delta &
    \Delta \end{bmatrix} \begin{bmatrix} \E N_i \\ \E Q_i
  \end{bmatrix}.
  \]
  The largest eigenvalue of the above square matrix is
  \[
  \Lambda = \frac{1 + \Theta + \Delta}{2} + \sqrt{ \Bigl( \frac{1
      + \Theta + \Delta}{2} \Bigr)^2 - \Delta},
  \]
  and we have $\E N_i + \E Q_i \leq C_0 \Lambda^i$ for some
  constant $C_0$. A similar argument as that in the proof of
  Theorem~\ref{the:dimupperboundold} gives that for all $\varepsilon >
  0$, almost surely, there exists a constant $C$ such that
  \[
  N_i + Q_i \leq C (1+\varepsilon)^i \Lambda^i.
  \]

  The set
  \[
  G_{n_l} = \bigcap_{j=n_l}^\infty \bigcup_{k = 1}^{n_j} B(\omega_k,
  r_{n_j})
  \]
  can be covered by the $N_i + Q_i$ balls of radius $r_{n_i}$. Hence
  \[
  \dimh G_{n_l} \leq \frac{\log (1 + \varepsilon) + \log
    \Lambda}{\log \theta}
  \]
  and since $\varepsilon$ can be taken as close to $0$ as we
  desire, we obtain $\dimh G_{n,l} \leq \frac{\log \Lambda}{\log
    \theta}$.  Which $\theta$ is the optimal choice depends on
  $c$.
\end{proof}

\subsection{Preparations to the proof of Theorem~\ref{the:dimlowerbound}}

We begin with the following theorem which is useful for the lower bound estimating of Hausdorff dimension.

Suppose $\eta$ is a Borel measure, and let $0 < s < 1$.  The Riesz
potential of $\eta$ is a function $R_s \eta \colon \mathbbm{T} \to
\mathbbm{R} \cup \{\infty\}$ defined by
\[
R_s \eta (x) = \int |x-y|^{-s} \, \mathrm{d}\eta (y).
\]

\begin{theorem} \label{lem:frostmanmeasure}
  Let $\eta$ be a finite Borel measure, and suppose $0 < s <
  1$. Then the Borel measure $\theta = (R_s \eta)^{-1} \eta$,
  defined by
  \[
  \theta (A) = \int_A (R_s \eta)^{-1} \, \mathrm{d} \eta,
  \]
  satisfies $\theta (U) \leq |U|^s$ for any Borel set $U$. ($|U|$
  denotes the diameter of $U$.)
\end{theorem}

\begin{proof}
  The proof follows the proof of Lemma~1.1 (or Lemma~5.1) in
  Persson \cite{Persson}.
  
  It is clear that $\theta$ is a Borel measure.

  We may assume that $\eta (U) > 0$, since otherwise, there is
  nothing to prove. It now follows that
  \begin{align*}
    \theta (U) &= \int_U \biggl( \int |x-y|^{-s} \, \mathrm{d}
    \eta (y) \biggr)^{-1} \, \mathrm{d} \eta (x) \\ &\leq \int_U
    \biggl( \int_U |x-y|^{-s} \, \mathrm{d} \eta (y) \biggr)^{-1}
    \, \mathrm{d} \eta (x) \\ & = \int_U \biggl( \int_U
    |x-y|^{-s} \, \frac{\mathrm{d} \eta (y)}{\eta (U)}
    \biggr)^{-1} \, \frac{\mathrm{d} \eta (x)}{\eta (U)} \\ &\leq
    \int_U \int_U |x-y|^{s} \, \frac{\mathrm{d} \eta (y)}{\eta
      (U)} \, \frac{\mathrm{d} \eta (x)}{\eta (U)} \leq |U|^s,
  \end{align*}
  where we have made use of Jensen's inequality.
\end{proof}

Throughout the proof of Theorem~\ref{the:dimlowerbound}, we will
assume that the balls $B_{k,n}$ are closed. This makes certain
arguments in the proof a bit simpler, and it does not change the
Hausdorff dimenstion of $\mathcal{U} (\omega)$.

We will consider a subset of $\UU(\omega)$. Let $n_j$ be a strictly
increasing sequence of integers. Put
\[
F_j = \bigcup_{k=n_{j-1} + 1}^{n_j} B_{k,n_{j+1}}.
\]
Then $F_j$ is compact and we have $F := \liminf F_j \subset
\UU(\omega)$. To see this, note that for every $n$ with $n_j < n
\leq n_{j+1}$, we have $F_j \subset E_n$, so that
\[
\bigcap_{j=j_0}^\infty F_j \subset \bigcap_{n=m}^\infty E_n
\]
holds when $n_{j_0} < m$.

We define measures $\mu_{l,m}$ on $\mathbbm{T}$ by
\[
\frac{\mathrm{d} \mu_{l,m}}{\mathrm{d} x} = \prod_{j=l}^m
\mathbbm{1}_{F_j}
\]
where $\mathbbm{1}_{F_j}$ denotes the indicator function of
$F_j$. Hence $\mu_{l,m}$ has support in
\[
\bigcap_{j=l}^m F_j.
\]
In fact, $\mu_{l,m}$ is the restriction of Lebesgue measure to
$\bigcap_{j=l}^m F_j$.

Suppose that $\{U_i\}$ is an open cover of $F$. Then $\{U_i\}$ is
an open cover of
\[
\bigcap_{j=l}^\infty F_j
\]
for any $l$. Since the sets $F_j$ are compact, there exists for
each $l$ an $m$, such that $\{U_i\}$ covers
\[
\bigcap_{j=l}^m F_j.
\]
We shall therefore investigate the typical behaviour of
$\mu_{l,m}$, aiming to use Theorem~\ref{lem:frostmanmeasure}.

We start with the following three lemmata.

\begin{lemma} \label{lem:tech}
  Let $x, y \in \mathbbm{T}$. Then
  \[
  \E \mathbbm{1}_{B_{k,n}} (x) \mathbbm{1}_{B_{k,n}} (y) = \int
  \mathbbm{1}_{B_{k,n}} (x) \mathbbm{1}_{B_{k,n}} (y) \,
  \mathrm{d} \Prob \leq 2 r_n \mathbbm{1}_{B (0, 2r_n)}
  (|x-y|).
  \]
\end{lemma}

\begin{proof}
  We have
  \[
  \int \mathbbm{1}_{B_{k,n}} (x) \mathbbm{1}_{B_{k,n}} (y) \,
  \mathrm{d} \Prob = \left\{ \begin{array}{ll} 0 & \text{if
    } |x-y| \geq 2 r_n \\ 2r_n - |x-y| & \text{if } |x-y| <
    2r_n \end{array} \right.,
  \]
  which proves the lemma.
\end{proof}

\begin{lemma} \label{lem:correlation}
  Let $x, y \in \mathbbm{T}$. Then
  \[
  \E (1 - \mathbbm{1}_{B_{k,n}} (x))(1 - \mathbbm{1}_{B_{k,n}}
  (y)) \leq 1 - 4 r_n + 2 r_n \mathbbm{1}_{B (0, 2r_n)} (|x-y|).
  \]
\end{lemma}

\begin{proof}
  We have
  \begin{align*}
   & \E (1 - \mathbbm{1}_{B_{k,n}} (x))(1 - \mathbbm{1}_{B_{k,n}}
    (y)) \\ =& 1 - \E \mathbbm{1}_{B_{k,n}} (x) - \E
    \mathbbm{1}_{B_{k,n}} (y) + \E \mathbbm{1}_{B_{k,n}} (x)
    \mathbbm{1}_{B_{k,n}} (y) \\ =& 1 - 4 r_n + \E
    \mathbbm{1}_{B_{k,n}} (x) \mathbbm{1}_{B_{k,n}} (y),
  \end{align*}
  and the estimate follows from Lemma~\ref{lem:tech}.
\end{proof}

\begin{lemma} \label{lem:Psi}
  Let
  \[
  \Psi_{l,m} (t) = \prod_{j=l}^m \biggl( 1 + \frac{(1 - 2
    r_{n_{j+1}})^{n_j-n_{j-1}} }{1 - (1 - 2 r_{n_{j+1}})^{n_j -
      n_{j-1}}} \mathbbm{1}_{B(0,r_{n_{j+1}})} (t)\biggr).
  \]
  If $n_j = \theta^j$ and $r_n = \frac{c}{n}$, then $\Psi_{l,m}
  (t) \leq 1 + C_l |t|^{-s(c,\theta)}$, where
  \[
  C_l = c^s \Bigl( 1 - \exp \Bigl(-2c \frac{\theta-1}{\theta^2}
  \Bigr) \Bigr)^l
  \]
  and
  \[
  s (c,\theta) = \frac{- \log \bigl(1 - \exp \bigl(-2c
    \frac{\theta-1}{\theta^2} \bigr) \bigr)}{\log \theta}.
  \]
\end{lemma}

\begin{proof}
  Take $t > 0$. If $t > r_{n_{l+1}}$, then $\Psi_{l,m} (t) = 1$.
  Otherwise, there is a $j_0 > l$ such that $r_{n_{j_0+1}}
  \leq t < r_{n_{j_0}}$. Then
  \begin{align*}
    \Psi_{l,m} (t) &= \prod_{j=l}^{j_0} \biggl( 1 + \frac{(1 - 2
      r_{n_{j+1}})^{n_j-n_{j-1}} }{1 - (1 - 2 r_{n_{j+1}})^{n_j -
        n_{j-1}}} \biggr) \\ &= \prod_{j=l}^{j_0} \biggl(
    \frac{1}{1 - (1 - 2 r_{n_{j+1}})^{n_j - n_{j-1}}} \biggr).
  \end{align*}
  With $n_j = \theta^j$ and $r_n = \frac{c}{n}$, we get
  \[
  \Psi_{l,m} (t) \leq \prod_{j=l}^{j_0} \Biggl(\frac{1}{ 1 - (1 -
    2c \theta^{- (j+1)})^{\theta^{j+1} (\theta - 1) \theta^{-2}}
  } \Biggr).
  \]    
  Since $x \mapsto (1-1/x)^x$ is increasing, we get
  \[
    \Psi_{l,m} (t) \leq \prod_{j=l}^{j_0} \Biggl(\frac{1}{1 -
      \exp \bigl(-2c \frac{\theta-1}{\theta^2} \bigr)} \Biggr) =
    C_l t^{-s(c,\theta)},
  \]
  where $C_l = c^s \bigl( 1 - \exp \bigl(-2c
  \frac{\theta-1}{\theta^2} \bigr) \bigr)^l$ and
  \[
  s (c,\theta) = \frac{- \log \bigl(1 - \exp \bigl(-2c
    \frac{\theta-1}{\theta^2} \bigr) \bigr)}{\log \theta}.
  \]

  Finally, we have $\Psi_{l,m} (t) \leq 1 + C_l
  |t|^{-s(c,\theta)}$, regardless of whether $t > r_{n_{l+1}}$ or
  not.
\end{proof}

We let
\[
K_{l,m} = \prod_{j=l}^m \bigl( 1 - (1 - 2 r_{n_{j+1}})^{n_j - n_{j-1}}
\bigr).
\]
These numbers will appear several times in our computations.

\begin{proposition} \label{prop:expectations}
  We have
  \[
  \E (\mu_{l,m} (\mathbbm{T})) = K_{l,m} \quad \text{and} \quad \E
  (\mu_{l,m} (\mathbbm{T})^2) \leq K_{l,m}^2 \int_\mathbbm{T}
  \int_\mathbbm{T} \Psi_{l,m} (|x-y|) \, \mathrm{d}x \mathrm{d}y.
  \]
\end{proposition}

\begin{proof}
  Since the intervals $[n_{j-1} + 1, n_j]$, which appear as
  summation intervals in the union
  \[
  F_j = \bigcup_{k=n_{j-1}+1}^{n_j} B_{k,n_{j+1}},
  \]
  are disjoint, the sets $F_j$ are pairwise independent, and we have
  \[
  \E (\mu_{l,m} (\mathbbm{T})) = \prod_{j=l}^m \int
  \int_{\mathbbm{T}} \mathbbm{1}_{F_j} (x) \, \mathrm{d}x \,
  \mathrm{d} \Prob = \prod_{j=l}^m \int_{\mathbbm{T}} \int
  \mathbbm{1}_{F_j} (x) \, \mathrm{d} \Prob \, \mathrm{d} x.
  \]

  We compute $\int \mathbbm{1}_{\complement F_j} \, \mathrm{d}
  \Prob$. By independence, we have
  \begin{align*}
    \int \mathbbm{1}_{\complement F_j} \, \mathrm{d} \Prob & =
    \int \prod_{k = n_{j-1} + 1}^{n_j} (1 -
    \mathbbm{1}_{B_{k,n_{j+1}}} ) \, \mathrm{d} \Prob\\
    & = \prod_{k
      = n_{j-1} + 1}^{n_j} \int (1 - \mathbbm{1}_{B_{k,n_{j+1}}}
    ) \, \mathrm{d} \Prob \\ &= \prod_{k = n_{j-1} + 1}^{n_j} (1
    - 2r_{n_{j+1}}) = (1 - 2 r_{n_{j+1}})^{n_j - n_{j-1}}.
  \end{align*}
  We then have
  \[
  \int \mathbbm{1}_{F_j} \, \mathrm{d} \Prob = 1 - (1
  - 2 r_{n_{j+1}})^{n_j - n_{j-1}}
  \]
  and
  \[
  \E (\mu_{l,m} (\mathbbm{T})) = \prod_{j=l}^m \bigl( 1 - (1 - 2
  r_{n_{j+1}})^{n_j - n_{j-1}} \bigr) = K_{l,m}.
  \]
  
  We now estimate $\E (\mu_{l,m} (\mathbbm{T})^2)$. By
  Lemma~\ref{lem:correlation}, we have
  \begin{align*}
    \int \mathbbm{1}_{\complement F_j} (x)
    \mathbbm{1}_{\complement F_j} (y) \, \mathrm{d} \Prob &=
    \prod_{k=n_{j-1}+1}^{n_j} \int (1 -
    \mathbbm{1}_{B_{k,n_{j+1}}} (x)) (1 -
    \mathbbm{1}_{B_{k,n_{j+1}}} (y)) \, \mathrm{d} \Prob \\ &\leq
    \bigl(1 - 4 r_{n_{j+1}} + 2 r_{n_{j+1}} \mathbbm{1}_{B(0, 2
      r_{n_{j+1}})} (|x - y|) \bigr)^{n_j - n_{j-1}} \\ &=:
    \Phi_j (|x-y|).
  \end{align*}
  Using this estimate, we have
  \begin{align*}
    \E (\mu_{l,m} &(\mathbbm{T})^2) \\ &= \int \biggl(
    \int_{\mathbbm{T}} \int_{\mathbbm{T}} \prod_{j=l}^m
    \mathbbm{1}_{F_j} (x) \mathbbm{1}_{F_j} (y) \, \mathrm{d} x
    \mathrm{d} y \biggr) \, \mathrm{d} \Prob \\ &=
    \int_{\mathbbm{T}} \int_{\mathbbm{T}} \prod_{j=l}^m \int (1 -
    \mathbbm{1}_{\complement F_j} (x)) (1 -
    \mathbbm{1}_{\complement F_j} (y)) \, \mathrm{d} \Prob \,
    \mathrm{d}x \mathrm{d}y \\ &= \int_{\mathbbm{T}}
    \int_{\mathbbm{T}} \prod_{j=l}^m \int \bigl( 1 -
    \mathbbm{1}_{\complement F_j} (x) - \mathbbm{1}_{\complement
      F_j} (y) + \mathbbm{1}_{\complement F_j} (x)
    \mathbbm{1}_{\complement F_j} (y) \bigr) \, \mathrm{d} \Prob
    \, \mathrm{d}x \mathrm{d}y \\ &\leq \int_{\mathbbm{T}}
    \int_{\mathbbm{T}} \prod_{j=l}^m \bigl( 1 - 2 (1 -
    2r_{n_{j+1}})^{n_j - n_{j-1}} + \Phi_j (|x-y|) \bigr) \,
    \mathrm{d}x \mathrm{d}y.
  \end{align*}
  We consider the factor $1 - 2 (1 - 2r_{n_{j+1}})^{n_j - n_{j-1}} +
  \Phi_j (|x-y|)$. If
  \[
  \mathbbm{1}_{B(0, 2 r_{n_{j+1}})} (|x - y|) = 0,
  \]
  then
  \begin{align*}
    1 &- 2 (1 - 2r_{n_{j+1}})^{n_j - n_{j-1}} + \Phi_j (|x-y|)
    \\ &= 1 - 2 (1 - 2r_{n_{j+1}})^{n_j - n_{j-1}} + (1 - 4
    r_{n_{j+1}})^{n_j-n_{j-1}} \\ &\leq 1 - 2 (1 -
    2r_{n_{j+1}})^{n_j - n_{j-1}} + (1 - 2
    r_{n_{j+1}})^{2(n_j-n_{j-1})} \\ &= \bigl(1 - (1 - 2
    r_{n_{j+1}})^{n_j-n_{j-1}} \bigr)^2.
  \end{align*}
  Similarly, if $\mathbbm{1}_{B(0, 2 r_{n_{j+1}})} (|x - y|) = 1$, then
  \begin{align*}
    1 &- 2 (1 - 2r_{n_{j+1}})^{n_j - n_{j-1}} + \Phi_j (|x-y|)
    \\ &= 1 - 2 (1 - 2r_{n_{j+1}})^{n_j - n_{j-1}} + (1 - 2
    r_{n_{j+1}})^{n_j-n_{j-1}} \\ &= 1 - (1 - 2
    r_{n_{j+1}})^{n_j-n_{j-1}}.
  \end{align*}
  Hence, we have
  \begin{multline*}
  \frac{1 - 2 (1 - 2r_{n_{j+1}})^{n_j - n_{j-1}} + \Phi_j
    (|x-y|)}{ \bigl(1 - (1 - 2 r_{n_{j+1}})^{n_j-n_{j-1}}
    \bigr)^2} \\ = \biggl(1 + \frac{(1 - 2
    r_{n_{j+1}})^{n_j-n_{j-1}}}{1 - (1 - 2
    r_{n_{j+1}})^{n_j-n_{j-1}}} \mathbbm{1}_{B(0, 2 r_{n_{j+1}})}
  (|x - y|) \biggr).
  \end{multline*}
  
  With
  \[
  \Psi_{l,m} (|x-y|) = \prod_{j=l}^m \biggl(1 + \frac{(1 - 2
    r_{n_{j+1}})^{n_j-n_{j-1}}}{1 - (1 - 2
    r_{n_{j+1}})^{n_j-n_{j-1}}} \mathbbm{1}_{B(0, 2 r_{n_{j+1}})}
  (|x - y|) \biggr)
  \]
  we therefore have
  \begin{align*}
    \E (\mu_{l,m} (\mathbbm{T})^2) & \leq K_{l,m}^2
    \int_\mathbbm{T} \int_\mathbbm{T} \Psi_{l,m} (|x-y|) \,
    \mathrm{d}x \mathrm{d}y. \qedhere
  \end{align*}
\end{proof}
  
\begin{proposition}  \label{prop:measures}
  Let $\varepsilon > 0$, $\delta > 0$, $\theta > 1$ and $n_j =
  \theta^j$. If $r_n = \frac{c}{n}$ with
  \[
  c > - \frac{1}{2} \frac{\theta^2}{\theta - 1} \log \Bigl(1 -
  \frac{1}{\theta} \Bigr),
  \]
  then we have
  \[
  \delta \E \mu_{l,m} (\mathbbm{T}) < \mu_{l,m} (\mathbbm{T}) <
  (2 - \delta) \E \mu_{l,m} (\mathbbm{T})
  \]
  with probability at least $1 - \varepsilon$ if $l$ is large enough.
\end{proposition}

\begin{proof}
  The assumption on $c$ implies that $s (c,\theta) < 1$.
  
  Using Lemma~\ref{lem:Psi} and Markov's inequality, we obtain
  \begin{align*}
    \Prob (|\mu_{l,m} (\mathbbm{T}) - \E \mu_{l,m}
    (\mathbbm{T})| \geq a) &\leq \frac{\E (\mu_{l,m}
      (\mathbbm{T}) - \E \mu_{l,m} (\mathbbm{T}))^2}{a^2} \\ &=
    \frac{\E ( \mu_{l,m} (\mathbbm{T})^2) - (\E \mu_{l,m}
      (\mathbbm{T}))^2}{a^2} \\ & \leq D_l \frac{(\E \mu_{l,m}
      (\mathbbm{T}) )^2}{a^2},
  \end{align*}
  where
  \[
  D_l := C_l \int_\mathbbm{T} \int_\mathbbm{T}
  |x-y|^{-s(c,\theta)} \, \mathrm{d}x \mathrm{d}y < \infty
  \]
  since $s (c,\theta) < 1$.
  
  With $a = (1 - \delta) \E \mu_{l,m} (\mathbbm{T})$, we get
  \[
  \Prob \Bigl( \delta < \frac{\mu_{l,m} (\mathbbm{T})}{\E
    \mu_{l,m} (\mathbbm{T})} < (2-\delta) \Bigr) \geq 1 -
  \frac{D_l}{(1- \delta)^2}.
  \]
  By Lemma~\ref{lem:Psi}, we see that $C_l \to 0$ and hence $D_l
  \to 0$ as $l \to \infty$. This finishes the proof.
\end{proof}

For $0 < s < 1$, we define the $s$-dimensional Riesz energy of a
measure $\mu$ by
\[
I_s (\mu) = \iint |x-y|^{-s} \, \mathrm{d}\mu (x) \mathrm{d}\mu(y).
\]
We let
\[
J_s = \int_\mathbbm{T} \int_\mathbbm{T} |x-y|^{-s} \, \mathrm{d}x
\mathrm{d}y = 2 \int_0^\frac{1}{2} t^{-s} \, \mathrm{d}t =
\frac{2^{2-s}}{1-s}.
\]

\begin{proposition} \label{prop:energy}
  Let $\theta > 1$ and $n_j = \theta^j$. If $r_n = \frac{c}{n}$
  with
  \[
  c > - \frac{1}{2} \frac{\theta^2}{\theta - 1} \log \Bigl(1 -
  \frac{1}{\theta} \Bigr),
  \]
  then
  \[
  \E I_s (\mu_{l,m}) \leq K_{l,m}^2 C_l J_{s+s(c,\theta)} <
  \infty.
  \]
\end{proposition}

\begin{proof}
  Following the same steps as in the estimation of $\E (\mu_{l,m}
  (\mathbbm{T})^2)$ in the proof of
  Proposition~\ref{prop:expectations}, we obtain
  \begin{align*}
    \E (I_s (\mu_{l,m})) &\leq \int_\mathbbm{T} \int_\mathbbm{T}
    |x-y|^{-s} \Psi_{l,m} (|x-y|) \, \mathrm{d}x \mathrm{d}y \\ &
    \leq K_{l,m}^2 C_l \int_\mathbbm{T} \int_\mathbbm{T}
    |x-y|^{-s-s(c,\theta)} \, \mathrm{d}x \mathrm{d}y = K_{l,m}^2
    C_l J_{s+s(c,\theta)}.
  \end{align*}
  The assumption on $c$ implies that $J_{s+s(c,\theta)}$ is finite.
\end{proof}

\subsection{Proof of Theorem~\ref{the:dimlowerbound}}

Assume that $\theta$ satisfies
\[
c > - \frac{1}{2} \frac{\theta^2}{\theta - 1} \log \Bigl(1 -
\frac{1}{\theta} \Bigr).
\]
Then $s(c,\theta) < 1$ and we may choose an $s \in (0, 1 -
s(c,\theta))$. Let $0 < \varepsilon < \frac{1}{2}$ and $\delta >
0$. We let $n_j = \theta^j$.

Our choice of $s$ implies that $J_{s+s(c,\theta)}$ is finite
since $s + s(c,\theta) < 1$. Using Markov's inequality and
Proposition~\ref{prop:energy}, we have
\[
\Prob (I_s (\mu_{l,m}) \geq a (\E \mu_{l,m})^2 ) \leq
\frac{\E I_s (\mu_{l,m})}{a (\E \mu_{l,m})^2} \leq
\frac{C_l J_{s+s(c,\theta)}}{a}.
\]
Take $a$ such that $\frac{C_l J_{s+s(c,\theta)}}{a} <
\varepsilon$. Then
\[
\Prob (I_s (\mu_{l,m}) < a (\E \mu_{l,m})^2 ) \geq 1 -
\varepsilon.
\]
Taking a large $l$, we deduce from Proposition~\ref{prop:measures}
that with probability at least $1 - 2 \varepsilon$,
\[
I_s (\mu_{l,m}) < a (\E \mu_{l,m})^2 \qquad \text{and} \qquad
\delta < \frac{\mu_{l,m} (\mathbbm{T})}{\E \mu_{l,m}
  (\mathbbm{T})} < 2 - \delta.
\]
Hence, with probability at least $1 - 2 \varepsilon$, for
each fixed $m > l$,
\[
I_s (\mu_{l,m}) < \frac{a}{\delta^2} (\mu_{l,m} (\mathbbm{T}))^2.
\]
We cannot guarantee that this holds with positive probability for
all $m > l$, but it follows that with probability at least $1 - 2
\varepsilon$, there is a sequence $(m_j)_{j=1}^\infty$, such that
for any $j$
\[
I_s (\mu_{l,m_j}) < \frac{a}{\delta^2} (\mu_{l,m_j} (\mathbbm{T}))^2.
\]
Suppose that $(\omega_k)_{k=1}^\infty$ is such that there exists
such a sequence $(m_i)$. We normalise $\mu_{l,m_i}$ by defining
the probability measure
\[
\nu_{l,m_i} = \frac{\mu_{l,m_i}}{\mu_{l,m_i} (\mathbbm{T})}.
\]
Then we may define measures $\theta_{l,m_i}$ by
\[
\frac{\mathrm{d} \theta_{l,m_i}}{\mathrm{d} \nu_{l,m_i}} = (R_s
\nu_{l,m_i})^{-1},
\]
where $R_s \nu_{l,m_i}$ is the $s$ dimensional Riesz potential.
By Theorem~\ref{lem:frostmanmeasure}
\[
\theta_{l,m_i} (U) \leq |U|^s,
\]
where $|U|$ denotes the diameter of $U$. By Jensen's inequality
we have
\[
\theta_{l,m_i} (\mathbbm{T}) \geq (I_s (\nu_{l,m_i}))^{-1} =
\frac{\mu_{l,m_i} (\mathbbm{T})^2}{I_s (\mu_{l,m_i})} \geq
\frac{\delta^2}{a}.
\]

Suppose now that $\{U_k\}$ is an open cover of $F$. Then
$\{U_k\}$ covers
\[
\bigcap_{j=l}^\infty F_j,
\]
for any $l$, and in particular for the large enough $l$ chosen
above. Since $F_j$ are compact, there is an $i$ such that
\[
\bigcup_k U_k \supset \bigcap_{j=l}^{m_i} F_j.
\]
Since ${U_k}$ covers the support of $\theta_{l,m_j}$, it follows that
\[
\sum_k |U_k|^s \geq \sum_k \theta_{l,m_i} (U_k) \geq
\theta_{l,m_i} (\mathbbm{T}) \geq \frac{\delta^2}{a}.
\]

The above proves that the $s$-dimensional Hausdorff measure of
$\bigcap F_j$ is at least $\delta^2 a^{-1}$ and in particular,
$\dimh F \geq s$ holds with probability at least $1 - 2
\varepsilon$. Since $s$ can be taken as close to $1 -
s(c,\theta)$ as we please, we therefore have proved that $\dimh F \geq 1
- s(c,\theta)$ holds with probability at least $1 - 2
\varepsilon$.

Since $\varepsilon > 0$ is arbitrary, by the fact that $\UU(\omega)
\supset F$, we deduce that
\[
\dimh \UU (\omega) \geq 1 - s(c,\theta) = 1 - \frac{- \log
  \bigl(1 - \exp \bigl(-2c \frac{\theta-1}{\theta^2} \bigr)
  \bigr)}{\log \theta}
\]
with probability $1$.


\begin{thebibliography}{00}

\bibitem{Bi65} P. Billard, \emph{S\'eries de Fourier
  al\'eatoirement born\'ees, continues, uniform\'ement
  convergentes}, Ann. Sci. \'Ecole Norm. Sup.  (4), 1965, 82:
  131--179.

\bibitem{Bu03} Y. Bugeaud, \emph{A note on inhomogeneous
  diophantine approximation}, Glasg. Math. J. 45 (2003):
  105--110.

\bibitem{Dr56} A. Dvoretzky, \emph{On covering a circle by
  randomly placed arcs}, Proc. Natl. Acad. Sci. USA %Proceedings of the National Academy of Sciences of the United States of America 
  42 (1956), 199--203.

\bibitem{Er61} P. Erd\H{o}s, \emph{Some unsolved problems},
  Publ. Math. Inst. Hung. Acad. Sci., Ser. A 6, (1961) 221--254.

\bibitem{FW04} A.-H. Fan and J. Wu, \emph{On the covering by
  small random intervals}, Ann. Inst. H. Poincar\'e
  Probab. Statist. 40(1):125--131, 2004.

\bibitem{Galambos} J. Galambos, \emph{The Asymptotic Theory of
  Extreme Order Statistics}, John Wiley and Sons, 1978.
  
\bibitem{Ka59} J.-P. Kahane, \emph{Sur le recouvrement d'un
  cercle par des arcs dispos\'es au hasard},
  C. R. Acad. Sci. Paris 248 (1959), 184--186.

\bibitem{Ka85} J.-P. Kahane, Some Random Series of Functions,
  Cambridge University Press, 1985.

\bibitem{Ka00} J.-P. Kahane, \emph{Random coverings and
  multiplicative processes}, Fractal geometry and stochastics, II
  (Greifswald/Koserow, 1998), 125--146, Progr. Probab., 46,
  Birkh\"auser, Basel, 2000.


\bibitem{KL19} D. H. Kim and L. Liao, \emph{Dirichlet uniformly
  well-approximated numbers}, Int. Math. Res. Not. IMRN %International Mathematics Research Notices, 
  24 (2019), 7691--7732.
  
\bibitem{KKP20}  M. Kirsebom, Ph. Kunde, T. Persson, \emph{Shrinking targets and eventually always hitting points for interval maps}, Nonlinearity 33 (2020), no. 2, 892--914.

\bibitem{Ma72} B. B. Mandelbrot, \emph{Renewal sets and random
  cutouts}, Z. Wahrsch. V. Geb.
%Zeitschrift ffir Wahrscheinlichkeitstheorie und Verwandte Gebiete 
  22 (1972), 145--157.

\bibitem{Or71} S. Orey, \emph{Random arcs on the circle},
  University of Minnesota, 1971.

\bibitem{Persson} T. Persson, \emph{Inhomogeneous potentials,
  Hausdorff dimension and shrinking targets}, Ann. H.
  Lebesgue, 2 (2019), 1--37.
  
\bibitem{Shepp} L. A. Shepp, \emph{Covering the circle with
  random arcs}, Israel J. Math. 11 (1972), 328--345.
  
\bibitem{ST03} S. Troubetzkoy, and J. Schmeling,
  \emph{Inhomogeneous Diophantine approximation and angular
    recurrence properties of the billiard flow in certain
    polygons}. Mat. Sb. 194, no. 2 (2003): 129--44. translation
  in Sb. Math. 194, no. 2 (2003): 295--309.
  
\bibitem{Wa12} M. Waldschmidt, \emph{Recent Advances in
  Diophantine Approximation.} Number Theory, Analysis and
  Geometry. 659--704. New York: Springer, 2012.

\end{thebibliography}
\end{document}